\documentclass[12pt]{article}

\usepackage{amsfonts,latexsym,amssymb,amsmath,amsthm,authblk,color,hyperref}

\date{\today}
\def \al{\alpha}
\def \be{\beta}
\def \ga{\gamma}
\def \dl{\delta}
\def \ep{\varepsilon}

\def \la{\lambda}
\def \si{\sigma}
\def \ph{\varphi}
\def \om{\omega}
\def \Ga{\Gamma}

\def \Om{\Omega}

\def \operatorname#1{\mathop{\rm #1}}

\def\div{\operatorname{div}}
\def\osc{\operatorname{osc}}
\def\osc2{\operatorname{osc^2}}

\def\supp{\operatorname{supp}}

\def\cd{\partial}
\def\esssup{\operatorname{esssup}}

\def\Q0{Q(x_0,t_0,R)}
\def\0{{x_0,t_0,R}}
\def\build#1_#2{\mathrel{\mathop{\kern 0pt#1}\limits_{#2}}}

\newtheorem{theorem}{Theorem}[section]

\newtheorem{proposition}{Proposition}[section]
\newtheorem{lemma}{Lemma}[section]
\newtheorem{definition}{Definition}[section]

\begin{document}
\title{On the existence and  uniqueness of   weak solutions \\ to     elliptic   equations with a singular drift}
\author{Misha Chernobai\footnote{Department of Mathematics, UBC, Canada, mchernobay@gmail.com}, Tim  Shilkin\footnote{V.A.~Steklov Mathematical Institute, St.-Petersburg, Russia,  and Leipzig University, Germany,  tim.shilkin@gmail.com}}

\maketitle
 \abstract{In this paper we study the Dirichlet problem for a scalar elliptic equation in a bounded Lipschitz domain $\Om \subset \Bbb R^3$ with a singular drift of the form  $b_0= b-\al \frac {x'}{|x'|^2}$ where $x'=(x_1,x_2,0)$, $\al \in \Bbb R$ is a parameter and $b$ is a divergence free vector field having essentially the same regularity as the potential part of the drift. Such drifts naturally  arise  in the theory  of  axially symmetric solutions to the Navier-Stokes equations. For $\al <0$ the divergence of such drifts   is positive  which potentially can ruin the uniqueness of   solutions. Nevertheless, for $\al<0$ we prove  existence and H\" older continuity of a unique  weak solution   which vanishes on the axis  $\Ga:=\{ ~x\in \Bbb R^3:~|x'|=0~\}$.}

\section{Introduction and Main Results}

\bigskip  Assume  $\Om \subset \Bbb R^3$ is a bounded domain     with the Lipschitz boundary  $\cd \Om$. Without loss of generality we can assume $\Om$ contains the origin.    We consider the
following boundary value problem:
\begin{equation}
\left\{ \quad \gathered     -\Delta u +  b_0 \cdot \nabla u \ = \ -\div f \qquad\mbox{in}\quad \Om, \\
u|_{\cd \Om} \ = \ 0.
\endgathered\right.
\label{Equation}
\end{equation}
Here   $u:  \Om\to \Bbb R$ is unknown,   $b_0:\Om\to \Bbb R^3$ and
$f: \Om \to \Bbb R^3$ are given functions.

In this paper we study the problem \eqref{Equation} for the drift $b_0$ of a special form. Our motivating example is
\begin{equation} b_0(x) =  -\al  \, \frac {x'}{|x'|^2}, \qquad x'=(x_1, x_2, 0), \qquad \al\in \Bbb R
\label{Typical_b}
\end{equation}
where   $\al \in \Bbb R$ is a  parameter.   Clearly, in this case the drift does not belong to $L_2(\Om)$, so instead we  assume that $b_0$ belongs to some critical
weak Morrey space
\begin{equation}
b_0\in L^{2, 1}_w(\Om), \label{Assumptions_Morrey_space_b_0}
\end{equation}
where
$L^{p,\la}_w(\Om)$ is the weak Morrey  space  equipped
with the quasinorm
$$
\| b_0\|_{L^{p,\la}_w(\Om)} \ := \ \sup\limits_{x_0\in
\Om}\sup\limits_{r<1} \, r^{-\frac \la p} \,
\|b_0\|_{L_{p,w}(B_r(x_0)\cap \Om)}
$$
and $L_{p,w}(\Om)$ is the weak Lebesgue space equipped with the quasinorm
\begin{equation}
\| b_0\|_{L_{p,w}(\Om)} \ := \ \sup\limits_{s>0}
s~|\{~x \in \Om: ~|b_0(x )|>s~\}|^{\frac 1p}.
\label{Weak_Lebesgue}
\end{equation}
We define the bilinear form $\mathcal B[u,\eta]$ by
\begin{equation}
\mathcal B[u,\eta] \ := \ \int\limits_{\Om}  \eta \,  b_0   \cdot \nabla u  ~dx. \label{Bilinear_Form}
\end{equation}
Note that for
$b_0$   satisfying \eqref{Assumptions_Morrey_space_b_0} the bilinear form   \eqref{Bilinear_Form}
generally speaking is not well-defined  for $u\in W^{1 }_2(\Om)$ where we denote by $W^k_p(\Om)$ the standard Sobolev space, see notation at the end of this section. Nevertheless, $\mathcal B[u,\eta]$  is
well-defined at least   for $u\in W^{1 }_p(\Om)$ with $p>2$ and
$\eta\in L_{\frac{2p}{p-2}}(\Om)$.  So,
instead of the standard notion of weak solutions from the energy
class $W^{1 }_2(\Om)$   we introduce the definition of $p$-weak
solutions to the problem \eqref{Equation},  see also the related definitions in \cite{Kang_Kim}, \cite{Tsai},  \cite{Kwon_1}:

\begin{definition}\label{p-weak}
Assume  $p> 2$, $b_0\in L_{p'}(\Om)$, $p'=\frac{p}{p-1}$, and  $f\in L_1(\Om)$.
 We say $u$ is a $p$-weak solution to the problem
\eqref{Equation} if $ u\in   \overset{\circ}{W}{^1_p}(\Om) $
and $u$ satisfies the identity
\begin{equation}
\gathered  \int\limits_{\Om}   \nabla u \cdot
\nabla\eta \, dx  \ + \ \mathcal B[u,\eta] \ = \
\int\limits_{\Om} f\cdot \nabla \eta~dx   , \qquad  \forall~\eta\in C_0^\infty(\Om).
\endgathered
\label{Identity}
\end{equation}
If  $b\in L_2(\Om)$  and  $u\in   \overset{\circ}{W}{^1_2}(\Om)$ satisfy  \eqref{Identity} then we call $u$ a weak solution to the problem  \eqref{Equation}. Obviously,  in this case  $p$-weak solutions are some subclass of weak solutions.

\end{definition}
Note that if $u$ is a weak or a $p$-weak solution to \eqref{Equation} and $f\in L_2(\Omega)$ then by density arguments we can extend the class of test functions in \eqref{Identity} from $\eta\in C_0^\infty(\Om)$ to all functions $\eta\in\overset{\circ}{W}{^1_2}(\Om)\cap L_{\infty}(\Omega)$.
We can consider two special cases of  the drift $b_0$:
\begin{equation}
 \div b_0  \, \le  \, 0 \quad \mbox{in} \quad
\mathcal D'(\Om), \label{Assumptions_b_good_sign}
\end{equation}
or
\begin{equation}
 \div b_0  \,  \ge  \, 0 \quad \mbox{in} \quad
\mathcal D'(\Om), \label{Assumptions_b_bad_sign}
\end{equation}
where we denote by $\mathcal D'(\Om)$ the space of distibutions on $\Om$.
The principal difference between  the cases \eqref{Assumptions_b_good_sign} and \eqref{Assumptions_b_bad_sign} is due to the fact that under the assumption \eqref{Assumptions_b_good_sign} the quadratic form $\mathcal B[u,u]$ provides (at least, formally) a positive support to the quadratic form of the elliptic operator in \eqref{Equation}, while in the case of \eqref{Assumptions_b_bad_sign}  the quadratic form $\mathcal B[u,u]$  is   non-positive and hence it ``shifts'' the operator to the ``spectral area''. For example, it is well-known that in the case \eqref{Assumptions_b_bad_sign}   the uniqueness for  the problem \eqref{Equation}
can be violated even for smooth solutions.
Indeed, the function $u(x) = c(1-|x|^2)$ is a solution to the problem \eqref{Equation} in the unite ball $\Om = \{ \, x\in \Bbb R^n: \,  |x|<1\, \}$ corresponding to $b_0(x) = n\, \frac{x}{|x|^2}$  with $b_0\in L_{n,w}(\Om)$ satisfying \eqref{Assumptions_b_bad_sign}, see also the discussion in      \cite{Chernobai_Shilkin}, \cite{Filonov_Shilkin}.

The case \eqref{Assumptions_b_good_sign} under the assumption \eqref{Assumptions_Morrey_space_b_0} was studied      in  \cite{Chernobai_Shilkin_1} (see also \cite{Chernobai_Shilkin} for the 2D case). In this paper we focus on the case \eqref{Assumptions_b_bad_sign} which is much more subtle and  for now we are not able to treat it in   full generality. So,  we restrict ourselves to the drifts with the potential part of some specific  form. Namely, we assume that
\begin{equation}\label{Drift_Specific}
b_0(x) \ = \ b(x) \, + \, \al \, \frac{x'}{|x'|^2},
\end{equation}
where $b: \Om \to \Bbb R^3$ satisfies the divergence-free condition (in the sense of distributions)
\begin{equation} \label{Assumption_b}
\div b = 0 \qquad \mbox{in}\quad \mathcal D'(\Om).
\end{equation}
Note that from \eqref{Assumptions_Morrey_space_b_0} we obtain
\begin{equation}
b\in L^{2, 1}_w(\Om). \label{Assumptions_Morrey_space_b}
\end{equation}
The relation \eqref{Drift_Specific} can be viewed as the Helmogoltz decomposition of the vector field $b_0$
\begin{equation}\label{Helmgoltz_decomposition}
b_0 \ = \ b + \al\, \nabla h  \quad \mbox{a.e. in} \quad \Om,
\end{equation}
where the potential part of the drift $h: \Om\to \Bbb R$   is specified by
\begin{equation}
  h(x) \, =   \, \ln \frac 1{|x'|}.
\label{Typical_h}
\end{equation}
In this case we have
\begin{equation}
-\Delta h  \ =  \  2\pi\,  \dl_\Ga \quad \mbox{in} \quad
\mathcal D'(\Om),  \qquad \Ga \, = \, \{\, x\in \bar \Om : \, x'=0 \, \},\label{h_is harmonic}
\end{equation}
where $\dl_\Ga $ is the delta-function concentrated on $\Ga$, i.e.
$$
\langle \dl_\Ga , \ph\rangle \ := \ \int\limits_\Ga \ph(x)\, dl_x, \qquad \forall\, \ph\in C_0^\infty(\Bbb R^3).
$$
Certainly, in the case of \eqref{Drift_Specific}  the identity \eqref{Identity} reduces  to
\begin{equation}
\int\limits_\Om \nabla u\cdot (\nabla \eta +b\eta)\, dx \, + \, 2\pi \, \al\, \int\limits_\Ga u(x)\eta(x)\, dl_x \, = \, \int\limits_{\Om} f\cdot \nabla \eta~dx , \quad \forall\, \eta\in C_0^\infty(\Om).
\label{Identity1}
\end{equation}
Note that    for  $u \in W^1_p(\Om)$ with $p>2$ the trace $u|_\Ga$ of $u$ on   $\Ga$ satisfies
\begin{equation}\label{Trace}
  u|_\Ga \in W^{1-\frac 2p}_p(\Ga),
\end{equation}
where   $W^s_p(\Ga)$, $s>0$, is the Slobodetskii-Sobolev space, see, for example,  \cite{BIN}. So, for  $p$-weak solutions in the sense of Definition \ref{p-weak}
 the second term in the left-hand side of  \eqref{Identity1} is well-defined.
Note also that the condition \eqref{Assumptions_b_bad_sign} corresponds to $\al\le 0$ in \eqref{Drift_Specific} and \eqref{Identity1}.

 The drift of type \eqref{Drift_Specific}  plays an important role in the theory of axially symmetric solutions to  the Navier-Stokes equations, see, for example,
\cite{KNSS},  \cite{Seregin_Book}, \cite{Seregin_Axi_1}, \cite{Seregin_Axi_2}, \cite{Seregin_Shilkin_UMN},
\cite{SS}, \cite{Tsai_Book}.
In the axially symmetric case the Navier-Stokes   system can reduced to the scalar equation
\begin{equation}
\cd_t u -\Delta u + \Big(v -\al \, \frac{x'}{|x'|^2}\Big)\cdot
\nabla u \ = \ 0 \qquad\mbox{in}\quad \Bbb R^3\times (0,T),
\label{NSE}
\end{equation}
where $v=v(x,t)$ is the divergence-free
velocity field  and $u= u(x,t)$ is some
auxiliary scalar function. For example, for axially symmetric
solutions without swirl (i.e. if $v(x,t) = v_r(r,z,t){  e}_r +
v_z(r,z,t){ e}_z$ where ${ e}_r$, ${ e}_\ph$, ${  e}_z$ is the
standard cylindrical basis) the equation \eqref{NSE} is satisfied
for $\al= 2$ and $u= \frac{\om_\ph}r$, where $\om_\ph:=v_{r,z} -
v_{z,r}$ and $r=|x'|$. In the case of general axially symmetric
solutions $v(x,t) = v_r(r,z,t){  e}_r + v_\ph(r,z,t){ e}_\ph +
v_z(r,z,t){  e}_z$ the equation \eqref{NSE} holds for $\al=-2$ and
$u=rv_\ph$.

It is well-known in the Navier-Stokes theory (see
\cite{KNSS},  \cite{Seregin_Book}, \cite{Seregin_Shilkin_UMN},
\cite{SS}, \cite{Tsai_Book}) that while in the case $\al> 0$ some
results like Liouville-type theorems assume no special conditions on
the solutions $u$ to the equation \eqref{NSE} besides a proper decay
of the drift $v$, the analogues results in the case $\al<0$ require
the additional condition $u|_{\Ga}=0$. Our equation \eqref{Equation} under the assumption \eqref{Drift_Specific} can be considered as
the elliptic model for the general equation \eqref{NSE}. The main  goals of the present paper is to investigate the equation \eqref{NSE} from the point of view of the ``general theory'' (i.e.  without the assumption on the axial symmetry of $u$ and other specific properties of solutions to the Navier-Stokes equations). In particular, we would like to clarify   the role  which the condition $u|_{\Ga}=0$  plays in the theory. On the other hand, our present contribution can be viewed as  a 3D extension of  the results  obtained earlier in   \cite{Chernobai_Shilkin} in the 2D case.

\bigskip
The  main result the present paper are the following two theorems:

\begin{theorem}\label{Theorem_3} Assume     $\Om\subset \Bbb R^3$ is a bounded Lipschitz domain containing the origin,  $b_0$ is given by \eqref{Drift_Specific} with $\al<0$,  $b$  satisfies   \eqref{Assumption_b}, \eqref{Assumptions_Morrey_space_b} and
$f\in L_q(\Om)$ with $q>3$.   Then  every $p$--weak solution to the problem \eqref{Equation} satisfying the condition
 \begin{equation} u|_{\Ga} \ = \ 0 \label{Zero_on_Gamma}
\end{equation}
is H\" older continuous. Namely,
there    exists    $\mu\in (0,1)$  depending only on     $q$,  $\al$, $\| b\|_{L^{2, 1}_w(\Om)}$ and the Lipschitz constant of $\cd \Om$
such that
if for some $p>2$ a function $u$ is a $p$--weak solution to the problem \eqref{Equation} corresponding to the right-hand side $f\in L_q(\Om)$  and satisfying the condition
 \eqref{Zero_on_Gamma}
in the sense of traces
 then    $u$ is H\" older continuous on $\bar \Om$ with the exponent $\mu$  and the  estimate
\begin{equation}
  \|u\|_{C^\mu(\bar \Om)} \ \le \ c~\| f\|_{L_q(\Om)},
\label{Holder_Estimate}
\end{equation}
 holds  with the
constant $c>0$ depending only on     $\Om$, $\al$,  $q$, $\|
b\|_{L^{2, 1}_{w}(\Om)} $ and the Lipschitz constant of $\cd \Om$.

\end{theorem}

We emphasize that the H\" older exponent $\mu$ in Theorem \ref{Theorem_3} does not depend on $p$, so  in this theorem we need the assumption $p>2$ only to have the trace in \eqref{Zero_on_Gamma} to be well-defined.

\begin{theorem}\label{Theorem_1} Assume     $\Om\subset \Bbb R^3$ is a bounded Lipschitz domain containing the origin,  $b_0$ is given by \eqref{Drift_Specific} with $\al<0$, $b$  satisfies   \eqref{Assumption_b}, \eqref{Assumptions_Morrey_space_b} and assume $q>3$. Then there exists  $p>2$
depending only on    $\al$, $q$,   $\| b \|_{L^{2, 1}_{w}(\Om)}$  and the Lipschitz constant of $\cd \Om$ such
that  for any $f\in L_q(\Om)$  there exists a unique $p$-weak
solution $u$ to
  the problem \eqref{Equation} satisfying the condition \eqref{Zero_on_Gamma}
   Moreover, this solution satisfies the estimate
   \begin{equation}
\| u\|_{W^1_p(\Om)} \ \le \ c~\| f\|_{L_q(\Om)},
\label{Main_Estimate}
\end{equation}
 with a
constant $c>0$ depending only on   $\Om$, $\al$, $q$, $\|
b\|_{L^{2, 1}_{w}(\Om)} $ and the Lipschitz constant of $\cd \Om$.
\end{theorem}

Note  that the condition \eqref{Zero_on_Gamma} is essential for the uniqueness in Theorem \ref{Theorem_1}, see the example of non-uniqueness for the Dirichlet problem in 2D case, for example,  in \cite{Chernobai_Shilkin}. Note also that for sufficiently regular drift $b_0$ the uniqueness for the Dirichlet problem \eqref{Equation} holds even without the condition \eqref{Zero_on_Gamma} as it follows from the maximum principle for the problem \eqref{Equation}, see, for example, \cite{Filonov_Shilkin} for the further discussion.

So, our main  message is: if the drift $b_0$ in \eqref{Equation} is singular  then the   ``bad'' sign of its divergence \eqref{Assumptions_b_bad_sign} can ruin the uniqueness for the problem \eqref{Equation}. But if the drift is singular along a particular curve $\Ga$ (and the singular part of the drift is harmonic away from this curve) then  the additional condition  \eqref{Zero_on_Gamma} (compensating  the singularity of the drift) provides  the existence and uniqueness in the class of $p$-weak solutions as well as guarantees some other ``good'' properties of solutions such as the H\" older continuity in Theorem \ref{Theorem_3} (see also the Liouville  property in   \cite{KNSS} and \cite{NU}). From our  point of view the most interesting result in Theorem \ref{Theorem_1} is the existence of solutions satisfying the condition \eqref{Zero_on_Gamma} as their  uniqueness     follows directly from the energy identity.

There are many papers devoted to the  investigation of the problem \eqref{Equation} in the case of divergence free drift, \cite{AD},  \cite{Tomek},
\cite{Filonov},
\cite{Filonov_Shilkin}, \cite{Fri_Vicol}, \cite{Lis_Zhang},  \cite{MV},
\cite{NU}, \cite{SSSZ}, \cite{Vicol}, \cite{Zhang},
\cite{Zhikov} and references there. Papers devoted to the
non-divergence free drifts are not so numerous (see  \cite{BO}, \cite{B}, \cite{Chernobai_Shilkin}, \cite{Kang_Kim},
\cite{Kim_Kim}, \cite{Kim_Kwon}, \cite{Tsai}, \cite{Krylov_1}, \cite{Krylov_2}, \cite{Krylov_3}, \cite{Kwon_1}, \cite{Kwon_2}, \cite{NU} for the references). Our  present  contribution  can be viewed as a 3D analogue of the results obtained earlier in \cite{Chernobai_Shilkin} in the 2D case.

Our paper is organized as follows. In Section \ref{Auxiliarly_Section} we introduce some auxiliary results and derive the estimates of the bilinear form which are based on Fefferman \cite{Fefferman} and  Chiarenza and Frasca \cite{Chia_Fra}  inequalities. In Section \ref{Proof_T2} we prove  a priori global boundedness of $p$--weak solutions to the problem \eqref{Equation} satisfying the condition \eqref{Zero_on_Gamma}. Note that this  result holds  for a supercritical drift $b$, but it is  heavily based on the boundary conditions in \eqref{Equation} and has not its analogue in the local setting, see the discussion in \cite{Filonov_Shilkin}. In Section  \ref{DG_section} we adopt the  De Giorgi technique to investigation of the problem with the singular drift of type \eqref{Drift_Specific}. The basic assumption which allows us to use the De Giorgi technique in more or less standard way is the density condition \eqref{Density_condition}. To show the validity of this condition for the points on the singular curve $\Ga$ we follow the method developed in \cite{CSTY}, see also \cite{Tsai_Book}. In Section \ref{Proof_T3} we prove Theorem \ref{Theorem_3}. Finally, in Section \ref{Proof_T1} we prove Theorem \ref{Theorem_1}.

\medskip

In the paper we use the following notation. For any $a$, $b\in
\mathbb  R^n$ we denote by $a\cdot b = a_k b_k$ their  scalar product in
$\mathbb R^n$. Repeated indexes assume the summation from 1 to $n$. An index after comma means partial derivative with respect to $x_k$, i.e. $f_{,k}:=\frac{\cd f}{\cd x_k}$.  We denote by $L_p(\Omega)$ and $W^k_p(\Omega)$ the
usual Lebesgue and Sobolev spaces. We do not distinguish between functional spaces of scalar and vector functions and omit the target space in notation.  $C_0^\infty(\Om)$ is the space of
smooth functions compactly supported in $\Om$. The space
$\overset{\circ}{W}{^1_p}(\Omega)$ is the closure of
$C_0^\infty(\Omega)$ in $W^1_p(\Omega)$ norm and $W^{-1}_p(\Om)$ is the dual space for $\overset{\circ}{W}{^1_{p'}}(\Omega)$, $p'=\frac p{p-1}$.  The space of
distributions on $\Om$ is denoted by $\mathcal D'(\Om)$. By  $C^\mu(\bar \Omega)$, $\mu\in (0,1)$ we denote
the spaces of   H\" older continuous functions on $\bar
\Omega$.
 The symbols $\rightharpoonup$ and $\to $ stand for the
weak and strong convergence respectively. We denote by $B_R(x_0)$
the ball in $\mathbb R^n$ of radius $R$ centered at $x_0$ and write
$B_R$ if $x_0=0$. We write   $B$ instead of $B_1$ and denote  $S_1:= \cd B_1$. For a domain $\Om\subset \Bbb R^n$ we also denote $\Om_R(x_0):=\Om\cap B_R(x_0)$.
 For $u\in L_\infty(\om)$ we denote   $$\operatorname{osc}\limits_{\om} u \,  := \, \esssup\limits_{\om}u - \operatorname{essinf}\limits_{\om}u.$$
We denote by $L^{p,\la}(\Om)$ the Morrey space equipped with the norm
$$
\| u \|_{L^{p, \la}(\Om)} \ := \ \sup\limits_{x_0\in \Om}\sup\limits_{R<\operatorname{diam}\Om}R^{-\frac \la p}\| u\|_{L_p( \Om_R(x_0))}.
$$
$(f)_\om$ stands for the average of $f$ over the domain $\om\subset \Bbb R^n$:
$$
(f)_\om \ := \ -\!\!\!\!\!\!\int\limits_\om f~dx \ = \ \frac
1{|\om|}~\int\limits_\om f~dx.
$$

\medskip\noindent
{\bf Acknowledgement.} This research  of Tim Shilkin was partly done during his stay  at the Max Planck Institute for Mathematics in the Sciences (MiS) in Leipzig in 2023. The author thanks  MiS for the hospitality. The research of Misha Chernobai was partially supported by Natural Sciences and Engineering Research Council of Canada (NSERC) grant RGPIN-2023-04534.

\newpage
\section{Auxiliary results}\label{Auxiliarly_Section}
\setcounter{equation}{0}

\bigskip
In this section we present several auxiliary results. Within  this section we assume that $\Om \subset \Bbb R^n$ is a bounded domain for arbitrary $n\ge 2$. The first result shows that by relaxing  an exponent of the integrability of a function we can always switch  from a weak Morrey norm   to a regular one.

\begin{proposition}\label{Holder_inequality}
  For any $p\in (1,n)$ and $1\le q<p\le r\le n$  there are positive constants $c_1$ and $c_2$ depending only on $n$, $p$, $q$ and $\Om$ such that
  $$
 c_1  \, \| b\|_{ L^{q,n-q}(\Om)} \, \le \,  \| b\|_{ L^{p,n-p}_w(\Om)}\, \le \,  c_2  \, \| b\|_{ L^{r,n-r}(\Om)}.
  $$
\end{proposition}

\begin{proof}
  The result follows from the H\" older inequality for Lorentz norms, see
\cite[Section 4.1]{Garfakos}.
\end{proof}

The next result is the estimate of the quadratic form corresponding to the drift term satisfying \eqref{Drift_Specific} and \eqref{Assumption_b}.

\begin{proposition}\label{Q_Form} Assume $p>2$, $p'=\frac p{p-1}$ and $b_0$ is given by \eqref{Drift_Specific} with $b\in L_{p'}(\Om)$ satisfying  \eqref{Assumption_b}.
Then for any $\al\in \Bbb R$ and any $u\in
 \overset{\circ}{W}{^1_p}(\Om)  $
the
 bilinear form
$\mathcal B [u,\eta]$ defined in \eqref{Bilinear_Form} satisfies the identity
\begin{equation}
\mathcal B [u,u] \ = \   \pi \al \,
 \int\limits_{\Ga}|u(x)|^2~dl_x
\label{Quadratic_Form}
\end{equation}
where the integral in the right-hand side is understood in the sense of traces.
\end{proposition}

\begin{proof}   For a smooth function  $u \in C_0^\infty(\Om)$  the relation \eqref{Quadratic_Form} follows by integration by parts. For an arbitrary function $u\in  \overset{\circ}{W}{^1_p}(\Om)  $ with   $p>2$ the corresponding relation follows from the continuity of the trace operator from $W^1_p(\Om)$ to $ L_p(\Ga)$.
\end{proof}

The next proposition proved by Chiarenza and Frasca in \cite{Chia_Fra} is the well-known extension of the result  of C.~Fefferman \cite{Fefferman} for $p=2$.
This theorem is one of basic tools in our proofs of both Theorems  \ref{Theorem_3} and \ref{Theorem_1}.

\medskip
\begin{proposition}
   \label{Morrey_usual}
   Assume $p\in (1, n)$, $r\in (1, \frac np]$ and  $V \in L^{r,n-pr} (\Om)$. Then
   $$
 \int\limits_\Om |V|\, |u|^p\, dx   \ \le \ c_{n,r,p}\, \| V\|_{L^{r,n-rp}(\Om)} \| \nabla u\|_{L_p(\Om)}^p , \qquad \forall\, u \in C_0^\infty(\Om).
   $$
  with the constant $c_{n,r,p}>0$ depending only on $n$, $r$ and $p$.

\end{proposition}

Our next result is the estimate of the bilinear form corresponding to the drift term. This result is a direct consequence of  Proposition \ref{Morrey_usual}.

\medskip
\begin{proposition}
   \label{Morrey_Compactness}
   Assume  $r\in \left( \frac {2n}{n+2}, 2\right) $ and  $b \in L^{r, n-r} (\Om)$. Then  there exists  $c>0$ depending only on $n$ and  $r$  such that    for any $u \in W^1_2(\Om)$ and    $\zeta \in C_0^\infty(\Om)$  the following estimate   holds:
   \begin{equation}\label{Morrey_higher_integrability}
  \int\limits_\Om \zeta^3 \, |b|\, |u |^2\, dx \ \le \ c\, \| b\|_{L^{r, n-r }(\Om)}  \,   \|\nabla (\zeta^2 u)\|_{L_2(\Om)} \,   \|\nabla (\zeta u)\|_{L_{\frac {2n}{n+2}}(\Om)}.
   \end{equation}
   Moreover, for any   $\theta\in \left(\frac nr -\frac n2,1\right)$   there exists  $c>0$ depending only on $n$, $r$ and $\theta$  such that    for any $u \in W^1_2(\Om)$ and  any $\zeta \in C_0^\infty(\Om)$  satisfying $0\le \zeta \le 1$ we have
  \begin{equation}\label{Morrey_Compactness_estimate}
 \int\limits_\Om \zeta^{1+\theta} \, |b|\, |u |^2\, dx \ \le \ c \, \| b\|_{L^{r, n-r }(\Om)} \|   u \|_{L_2(\Om)}^{1-\theta}   \|\nabla (\zeta u)\|_{L_2(\Om)}^{1+\theta}\, |\Om|^{\theta/n}.
  \end{equation}
  If we additionally assume $u\in \overset{\circ}{W}{^1_2}(\Om)$ then the estimates \eqref{Morrey_higher_integrability} and \eqref{Morrey_Compactness_estimate} remains true for an arbitrary $\zeta \in C_0^\infty(\Bbb R^n)$.
\end{proposition}

Proposition \ref{Morrey_Compactness} is proved in \cite[Proposition 3.4]{Chernobai_Shilkin_1}.

\newpage
\section{Boundedness of weak solutions}\label{Proof_T2}
\setcounter{equation}{0}

\bigskip
In this section we establish global boundedness of $p$--weak solutions to the problem \eqref{Equation}. Note that this result holds for a supercritical drift $b$, so we do not need the critical condition \eqref{Assumptions_Morrey_space_b} in this section.  Note also, that our result is global, i.e. it is heavily based on the homogeneous (or, more generally, regular) Dirichlet boundary conditions in \eqref{Equation} and this result is not  valid in the local setting, see, for example, counterexamples and discussion in \cite{Filonov_Shilkin}. For the related results in local setting we see the recent paper \cite{AD} and reference there.

\medskip
\begin{theorem}\label{Theorem_2}
 Assume $p>2$, $p'=\frac p{p-1}$ and $b_0$ is given by \eqref{Drift_Specific} with $b\in L_{p'}(\Om)$ satisfying  \eqref{Assumption_b}. Assume $q>3$ and $f\in L_q(\Om)$.
Then for any $\al\in \Bbb R$ any $p$-weak
solution $u\in
 \overset{\circ}{W}{^1_p}(\Om)  $ to
  the problem \eqref{Equation} satisfying  \eqref{Zero_on_Gamma}  is essentially bounded and  satisfies  the estimate
\begin{equation}
\| u\|_{L_\infty(\Om)} \ \le \ c~\| f\|_{L_q(\Om)},
\label{L_infty_Estimate}
\end{equation}
 with a
constant $c>0$ depending only on    $\Om$ and $q$.
\end{theorem}

\begin{proof}
For $m>0$   we define a truncation $T_m:\Bbb R\to \Bbb R$ by $T_m(s):= m$ for $s\ge m$, $T_m(s)=s$ for $s< m$.
For any $s\in \Bbb R$ we also denote $(s)_+:=\max\{s, 0\}$.
Now we fix some $m>0$ and
and denote $ \bar u := T_m(u)$.
  Then for any $k\ge 0 $ we have
  $$
  (\bar u -k)_+\in L_\infty(\Om)\cap \overset{\circ}{W}{^1_p}(\Om), \qquad \nabla (\bar u-k)_+ = \chi_{\Om [k<u< m] } \nabla u
  $$
where $\Om [k<u<m]=\{\, x\in\Om: \, k<u(x)<m \, \}$.  Define  $\eta := (\bar u -k)_+$  and note that from \eqref{Zero_on_Gamma} for any $m\ge 0$ and any $k\ge 0$ we obtain
\begin{equation}\label{eta_is_zero}
(u-m)_+|_\Ga =0, \qquad \eta|_{\Ga} \, = \, 0
\end{equation}
in the sense of traces.
Approximating  $\eta $ by smooth functions we can  take $\eta$ as a test function  in \eqref{Identity}. For $k\ge m$ we have $\eta\equiv 0$ and hence $\mathcal B[u, \eta]=0$. For $k< m$ we obtain
  $$
  \gathered
  \mathcal B[u, \eta] \ = \
   \mathcal B[\eta , \eta] \ + \ (m-k)\, \int\limits_{\Om   } b_0\cdot \nabla (u-m)_+ \, dx .
  \endgathered
  $$
  From Proposition \ref{Q_Form} taking into account \eqref{eta_is_zero} we obtain $\mathcal B[\eta, \eta] =0$. On the other hand, from \eqref{Assumption_b} and \eqref{h_is harmonic} taking into account \eqref{eta_is_zero} we obtain
  $$
   \int\limits_{\Om   } b_0\cdot \nabla (u-m)_+ \, dx \ = \ 2\pi \al\, \int\limits_\Ga (u-m)_+(x)\, dl_x  \ = \ 0
   $$
  and hence $\mathcal B[u,\eta]=0$.
So,  for $A_k:=\{ \, x\in \Om: \, \bar u(x)>k \, \}$  from   \eqref{Identity} we obtain
 $$
 \int\limits_\Om |\nabla (\bar u-k)_+|^2\, dx \ \le \ \| f\|_{L_q(\Om)}^2 |A_k|^{1-\frac 2q}, \qquad \forall\, k\ge 0,
 $$
 which implies (see   \cite[Chapter II, Lemma 5.3]{LU})
 $$
 \esssup\limits_\Om \bar u \ \le \ c \, |\Om|^{ \dl } \, \| f\|_{L_q(\Om)}, \qquad \dl:= \tfrac 13-\tfrac 1q,
 $$
 with some constant $c>0$ depending only on $q$.
 As this estimate is uniform with respect to $m>0$ we conclude  $u$ is essentially bounded from above and
 $$
 \esssup\limits_\Om  u \ \le \ c \, |\Om|^{ \dl } \, \| f\|_{L_q(\Om)}.
 $$
 Applying the same procedure to $\bar u:= T_m(-u)$  instead of $\bar u:= T_m(u)$  we obtain $u\in L_\infty(\Om)$ as well as \eqref{L_infty_Estimate}.
 \end{proof}

\newpage
\section{De Giorgi classes}\label{DG_section}
\setcounter{equation}{0}

\bigskip
\medskip

In this section we introduce the modified De Giorgi classes which are convenient for the study of solutions to the elliptic equations with coefficients from Morrey spaces.  These classes were used before  in \cite{Chernobai_Shilkin_1}. In this section we use the following notation: for $u\in L_\infty(\Om)$ and $B_\rho(x_0)\subset \Om$ we denote
  $$
  m(x_0, \rho) := \inf\limits_{B_\rho(x_0)} u, \qquad M(x_0, \rho) := \sup\limits_{B_\rho(x_0)} u,  \qquad \om(x_0, \rho):= M(x_0,\rho) -m(x_0,\rho) .
  $$

\begin{definition}\label{Def_DG} Assume $\Om\subset \Bbb R^n$ is a bounded Lipschitz domain, $n\ge 2$.
 For $u\in W^1_2(\Omega)$ we define $A_k:= \{\, x\in \Om: \, u(x)>k\, \}$. We say $u\in DG(\Om, k_0)$ if there exist  constants $\ga>0$, $F>0$, $\be\ge 0$, $q>n$ such that for any $B_R(x_0)\subset \Om$, any $0<\rho <R $ and any $k\ge k_0$ the following inequality holds
\begin{equation}
\gathered
  \int\limits_{A_k\cap B_\rho(x_0)} |\nabla u|^2\, dx \ \le \ \frac{\ga^2}{(R-\rho)^2}\left( 1+  \frac{R^\be}{(R-\rho)^\be}\right) \, \int\limits_{A_k\cap B_R(x_0)}|u-k|^2\,dx \ + \\ + \  F^2 \, |A_k\cap B_R(x_0)|^{1-\frac 2q} \label{DG}.
\endgathered
\end{equation}

\end{definition}

 To avoid overloaded notation,  when  we need to specify constants in Definition \ref{Def_DG} we allow some terminological license and say that  the class $D(\Om,k_0)$ corresponds to the constants $\ga$, $F$, $\be$, $q$  instead of including these constants  in the notation of the functional class.

The main result of this section is the following weak version of the maximum principle:

\begin{proposition}
  \label{Density_lemma} Let $\Om\subset \Bbb R^n$ be a bounded Lipschitz domain, $n\ge 2$.
 Denote by  $DG(\Om,    k_0)$ the De Giorgi class with parameters $\ga$, $\be$, $q$, $F$ and assume
  $u\in  DG(\Om,    k_0)$.   Then $u$ is locally essentially bounded from above in $\Om$  and for any $\dl\in (0,1)$ there exists $\theta \in (0,1)$ depending only on $\dl$, $\ga$, $\be$, $q$ such that for any $B_{4R}(x_0)\subset \Om$ if
\begin{equation}\label{Density_condition}
  \big| \{ \, x\in B_{2R}(x_0): ~u(x) \le k_0  \, \} \big|\ \ge
\ \dl \, |B_{2R}|
\end{equation}
 then
  \begin{equation}\label{maximum_decay_estimate}
  \sup\limits_{B_R(x_0)} u  \, \le \, (1-\theta) \, \sup\limits_{B_{4R}(x_0)} u  \, + \, \theta \, k_0 \,  + c_1\,  F\, R^{1-\frac nq},
  \end{equation}
where $c_1>0$ depends only on $n$, $\ga$, $\be$, $q$.
\end{proposition}

Proposition \ref{Density_lemma} is a simple combination of statements of Lemmas \ref{Boundedness_Holder} ---  \ref{How_to_find} below. These  Lemmas are standard and  their proofs can be found, for example,  in \cite{Chernobai_Shilkin_1}, see also \cite{LU}.

\begin{lemma}\label{Boundedness_Holder}
Assume  $u\in DG(\Om, k_0)$. Then $u$ is locally essentially bounded from above in $\Om$ and for any $B_{2R}(x_0)\subset \Om$
\begin{equation}\label{Boundedness_Estimate_Holder}
\sup\limits_{B_R(x_0) } ~(u-k_0)_+ \ \le \ c_*  \,  \left[\Big(\ \ \
-\!\!\!\!\!\!\!\!\!\!\!\!\int\limits_{B_{2R}(x_0) } |(u-k_0)_+|^2~dx \Big)^{1/2} + F\, R^{1-\frac nq}\right]
\end{equation}
where $(u-k_0)_+:=\max\{ u-k_0,0\}$ and  $c_*>0$ depends only on $n$, $\ga$ and $\be$.
\end{lemma}

\begin{lemma}\label{Thin_set}
Assume $u\in DG(\Om, k_0)$. Then there exists
$ \delta_0\in (0,1)$ depending on $  n$, $\ga$, $\be$ in Definition \ref{Def_DG} of the De Giorgi class such that for any
$  B_{2R}(x_0) \subset\Om$  if
$$
|B_{2R}(x_0)\cap A_{k_0}| \, \le \,  \delta_0 \, |B_{2R}|$$ then either
\begin{equation}\label{Sup_decay}
  \sup\limits_{B_R(x_0)} (u-k_0)_+ \ \le \ \frac 12~
\sup\limits_{B_{4R}(x_0)} (u-k_0)_+
\end{equation}
or
\begin{equation}\label{Sup_decay_1}
 \sup\limits_{B_{4R}(x_0)} (u-k_0)_+  \ \le \ 4c_* \, F \, (2R)^{1-\frac nq}
\end{equation}
where $c_*>0$ is a constant from \eqref{Boundedness_Estimate_Holder}.
\end{lemma}

\begin{lemma} \label{How_to_find}
Assume
$u\in DG(\Om, k_0)$. Then for any  $\dl\in (0,1)$  there exists  $s\in \Bbb N$ depending only on $\dl$, $n$, $\ga$, $\be$   such that  if for some $ B_{4R}(x_0)\subset\Om$  we have
$$
|B_{2R}(x_0)\setminus A_{ k_0}| \
\ge \ \dl~|B_{2R}|
$$
  then either
\begin{equation}\label{How_to_find_est}
|B_{2R}(x_0) \cap
A_{\bar k} | \ \le  \ \dl_0~ |B_{2R}|,
\end{equation}
or
\begin{equation}\label{How_to_find_est_1}
\sup\limits_{B_{4R}(x_0)} (u -k_0)_+ \, \le \,  2^s\, F\, R^{1-\frac nq}.
\end{equation}
Here $\dl_0\in (0,1)$ is the constant from Lemma \ref{Thin_set} and we denote
$$
\bar k \, = \,  M(x_0, 4R)  - \frac 1{2^s}\Big(M(x_0, 4R)  - k_0\Big), \qquad M(x_0,4R) \,    := \,
\sup\limits_{B_{4R}(x_0)} u.
$$
\end{lemma}

Proposition \ref{Density_lemma} provides the control of the oscillation of a function belonging to the modified De Giorgi class if the assumption \eqref{Density_condition} is satisfied. The validity of \eqref{Density_condition} for the points $x_0$ on the singular  curve $\Ga$ of the drift \eqref{Drift_Specific} follows from the following weak form of the Harnak inequality which we borrow from \cite{Tsai_Book}, see also \cite{Chernobai_Shilkin}. From now on we restrict ourselves to the case $n=3$ so that  the trace $u|_{\Ga}$ for a $p$-weak solution $u\in W^1_p(\Om)$ is well-defined.

\begin{proposition}\label{Weak_Harnack} Assume $B_R:=\{x\in \Bbb R^3: \, |x|<R\}$.
 Assume  $p>2$,   $b\in L_{p'}(B_2)$, $p'=\frac{p}{p-1}$ satisfies $\div b=0$ in $\mathcal D'(B_2)$,   and $g\in L_p(B_2)$. Assume
  $v\in  W^1_p(B_2)$ satisfies  the equation (in the sense of distributions)
  \begin{equation}
  -\Delta v +  b_0 \cdot \nabla v \, = \, -\div g \quad \mbox{in} \quad B_2, \qquad b_0 :=  b -\al \, \frac{x'}{|x'|^2} .
  \label{Equation_v}
  \end{equation}
 Assume $\al\not =0$ and
\begin{equation}\label{Assumptions_v}
  0\le v\le 2 \quad \mbox{in} \quad B_2, \qquad v|_{\Ga} \ge 1, \qquad \Ga:= \{ \, x\in \Bbb R^3:~ x_1=x_2=0~\}.
\end{equation}
  Then there exists constants $\dl_1 \in (0,1)$ and $\la_1\in (0,1)$ depending only on $\| b\|_{L_{p'}(B_2)}$ in the explicit way specified below such that if
\begin{equation}\label{Assumption_g}
  \| g\|_{L_2(B_2)} \, \le \, c_{\star} \, |\al|, \qquad \mbox{where} \qquad c_{\star} :=  \frac{\pi}{4|B_2|^{1/2}} ,
\end{equation}
  then
 $$
  \big| \{ \, x\in B_2: ~v(x) \ge \la_1  \, \} \big|\ \ge
\ \dl_1.
  $$
\end{proposition}

\begin{proof}  Assume there exist $g\in L_p(B_2)$ and $v\in W^1_p(B_2)$
satisfying \eqref{Equation_v}, \eqref{Assumptions_v},
\eqref{Assumption_g} such that
\begin{equation}
   |\{ \, x\in B_2 : \, v(x)> \la_1 \,
\}| \ \le \ \dl_1.
 \label{le_lambda}
\end{equation}
Multiplying \eqref{Equation_v} by an arbitrary $\eta\in
C_0^\infty(B_2)$ and integrating by parts    we obtain
\begin{equation}
 2\pi \al \int\limits_{\Ga\cap B_2} v(x)\eta(x)\, dl_x \ = \  \int\limits_{B_2 } v\Big(\Delta \eta +
b \cdot \nabla \eta  \Big)~dx \ + \ \int\limits_{B_2 } g\cdot
\nabla \eta~dx. \label{v(0)_estimate}
\end{equation}
Note that $v|_\Ga \ge 1$. Choose  $\eta \in C_0^\infty(B_2)$  so that
$$
\eta = 1 \quad \mbox{on} \quad B ,  \qquad \|\nabla\eta\|_{L_2(B_2)} \, \le \,  4\,  |B_2|^{1/2}, \qquad \|
\eta\|_{C^2(\bar B_2)}\, \le \, c_{\star\star},
$$
where $c_{\star\star}>0$ is
some sufficiently large absolute constant. Then from
\eqref{Assumption_g} and the H\" older inequality we obtain
$$
\Big| ~\int\limits_{B_2 } g\cdot \nabla \eta ~dx~\Big|  \ \le \ \pi
|\al|.
$$
Hence from \eqref{v(0)_estimate} we obtain
\begin{equation}
\pi |\al|   \ \le \  \Big| \, \int\limits_{B_2 } v\Big(\Delta \eta  +
b \cdot \nabla \eta   \Big)~dx ~\Big|. \label{together}
\end{equation}
Denote
$$
B_2 [v> \la_1] := \{~x\in B_2 : \, v(x)> \la_1~\}, \qquad B_2 [v\le \la_1] :=
\{~x\in B_2 : \, v(x)\le \la_1~\}.
$$
From the H\" older inequality we obtain
$$
 \int\limits_{B_2 [v> \la_1]} v\Big(\Delta \eta  +
b^{(\al)}\cdot \nabla \eta   \Big)~dx \  \le   \   \|
v\|_{L_\infty(B_2 )} \, \| \eta \|_{C^2(\bar B_2 )} \,
\Big(1+\|b \|_{L_{p'}(B_2 )} \Big) |B_2 [v>\la_1]|^{\frac
1{p}}.
$$
Taking into account \eqref{Assumptions_v}, \eqref{le_lambda}  and
$\|\eta\|_{C^2(B_2)}\le c_{\star\star}$  we conclude
$$
\Big| \int\limits_{B_2 [v> \la_1]} v\Big(\Delta \eta  + b  \cdot
\nabla \eta  \Big)~dx \, \Big|  \  \le   \  2c_{\star\star} \,
\Big(1+\|b \|_{L_{p'}(B_2 )} \Big) \dl_1^{\frac 1p}.
$$
  On the other hand
$$
\gathered  \Big|\int\limits_{B_2 [v\le \la_1]} v\Big(\Delta \eta  +
b \cdot \nabla \eta  \Big)~dx \Big| \  \le \ c_{\star\star}
\, \Big(1+\|b \|_{L_{1}(B_2 )} \Big) \, \la_1.
\endgathered
$$
Finally, we obtain
$$
|\al|   \  \le   \  2c_{\star\star} \,
\Big(1+\|b \|_{L_{p'}(B_2 )} \Big)  \dl_1^{\frac 1p} \ + \
c_{\star\star} \, \Big(1+\|b  \|_{L_{1}(B_2 )} \Big) \, \la_1.
$$
This inequality leads to  the contradiction if we fix values of
$\la_1$, $\dl_1 \in (0,1) $ so that
$$
 2c_{\star\star} \,
\Big(1+\|b \|_{L_{p'}(B_2 )} \Big) \dl_1^{\frac 1p} \ + \
 c_{\star\star} \, \Big(1+\|b  \|_{L_1(B_2 )}  \Big) \, \la_1  \ < \ \pi
|\al|.
$$
\end{proof}

\begin{proposition}\label{Weak_Harnack_1}  Assume $\Om\subset \Bbb R^3$ is a bounded domain which contains the origin. Denote by $\dl_1\in (0,1)$ and $\la_1\in (0,1)$ the constants from Proposition \ref{Weak_Harnack}.
  Assume $b\in L_{w}^{2,1}(\Om)$    satisfies $\div b =0$ in $\mathcal D'(\Om)$. Assume         $f\in L_p(\Om)$ with  $p>2$ and  let $u$ be a $p$-weak solution to the problem \eqref{Equation} such that
  $$
  u|_{\Ga } \, = \, 0.
  $$
  Assume  $\al\not =0$, $x_0\in \Ga$ and $B_{4R}(x_0)\subset \Om$.  Denote
    \begin{equation}\label{Define_k_0}
  k_0 \ := \   \frac12\, \Big(M(x_0, 4R)   +   m(x_0,4R)\Big),
  \end{equation}
  and
 \begin{equation}\label{Define_k_1}
  k_1 \ := \   M(x_0, 4R) \, - \,  \frac{\la_1}2  \, \om(x_0,4R) .
  \end{equation}
   If  $
  k_0   \, \ge \, 0
  $
  then either
   \begin{equation}\label{What we need}
  \big| \{ \, x\in B_{2R}(x_0): ~u(x) \le  k_1  \, \} \big|\ \ge
\ \dl_1 \, |B_{2R}|
  \end{equation}
   or
   \begin{equation}\label{Assumption_f}
  \om(x_0, 4R) \, \le \, \frac{2 R^{-1/2}}{c_\star |\al|} \, \| f\|_{L_2(B_{4R}(x_0))}
   \end{equation}
   where $c_\star$ is defined in \eqref{Assumption_g}.

\end{proposition}

\begin{proof} Assume \eqref{Assumption_f} does not hold, i.e.
 \begin{equation}\label{Assumption_not_f}
 R^{-1/2}  \, \| f\|_{L_2(B_{4R}(x_0))} \, \le \, \frac{c_\star |\al|}2 \, \om(x_0, 4R).
   \end{equation}
  For $x\in B_4$ we denote
$$
u^R(x)   =   u(x_0+ Rx), \qquad b_0^R(x) =    R\, b_0(x_0+ Rx), \qquad f^R(x)   =  R f(x_0+Rx) .
$$
Then $u^R$ is a solution to
$$
-\Delta u^R + b_0^R \cdot \nabla u^R \ = \ -\div  f^R
\qquad\mbox{in} \quad B_2
$$
and, moreover,
$$
\| b_0^R\|_{L_{p'}(B_2)}  \ \le \ c\, \| b_0\|_{L_w^{2,1}(\Om)}
$$
with some $c>0$ independent of $R$.
Define   $v$ and $g$ so that
$$
\gathered  v(x) \  := \ 2~\frac{M(x_0, 4R) -u^R (x)}{\om(x_0, 4R)}, \qquad g(x) \  := \ -2~\frac{f^R (x)}{\om(x_0, 4R)}.
\endgathered
$$
From \eqref{Assumption_not_f} we conclude that $g$ satisfies \eqref{Assumption_g}. Moreover,  $v$ is a $p$-weak solution to the equation \eqref{Equation_v} which satisfies
$$
0\le v \le 2, \qquad v|_{\Ga } \, \ge \, 1.
$$
Hence
 $$
  \big| \{ \, x\in B_2: ~v(x) \ge \la_1  \, \} \big|\ \ge
\ \dl_1 
  $$
  which gives \eqref{What we need}.
\end{proof}

Now we can prove the estimates of oscillation for functions belonging to various  modified De Giorgi classes. We will distinguish between the following three cases which  are motivated by the properties of  a $p$-weak solution $u$  to the problem \eqref{Equation} (we will proof this properties  in Section \ref{Proof_T3}):
\begin{itemize}
  \item  Assume $x_0$ is an internal point away from the singular line $\Ga$. In this case we assume  $B_R(x_0)\subset \Om \setminus \Ga$. In this case we will show that   $\pm u \in DG(B_R(x_0); k_0)$ for any $ k_0 \in \Bbb R$, i.e. away from the singular line $\Ga$ a $p$-weak solution to the problem \eqref{Equation} belongs to the De Giorgi class with arbitrary starting level $k_0$. In this case the estimate of the oscillation of $u$ in the ball $B_R(x_0)$ is standard.

  \item Assume $x_0$ is an internal point which belongs to the singular line $\Ga$. In this case we assume  $B_R(x_0)\subset \Om$ and  $x_0\in \Ga$. Due to the assumption \eqref{Zero_on_Gamma} this case would correspond to the condition     $\pm u \in DG(\Om; 0)$ (i.e. the De Griorgi class $DG(\Om; k_0)$ with the fixed starting level  $k_0=0$).  In this case the estimate of the oscillation of $u$ in the ball $B_R(x_0)$ is based on Proposition \ref{Weak_Harnack_1}.

  \item Finally, consider a boundary point $x_0\in \cd \Om$ (including the case when $x_0\in \cd\Om\cap \Ga$).
   In this case we take arbitrary   $\Om_0$  such that $\Om \Subset \Om_0 $ and  denote by $ \bar u $   the zero extension of $u$ onto $\Om_0\setminus \Om$.  So, this case
    also  corresponds to the condition
    $\pm \bar  u \in DG(\Om_0; 0)$ and due to the Dirichlet condition in \eqref{Equation} and the Lipschitz continuity of $\cd \Om$ we may  assume  that $\tilde u$ vanishes on a fixed  portion of $B_R(x_0)\subset \Om_0$ (so the density condition \eqref{Density_condition} in this case is satisfied for free because of the Dirichlet condition in \eqref{Equation}).

\end{itemize}

We start from the  oscillation estimate for the non-singular internal points. The proof of this estimate is standard and we present it only for readers' convenience.

\begin{lemma}\label{osc_lemma}  Let $\Om\subset \Bbb R^n$ be a bounded Lipschitz domain, $n\ge 2$, and $\pm u\in DG(\Om, k_0)$ for any $k_0\in \Bbb R$, where $DG(\Om; k_0)$ is the De Giorgi class with the parameters  $\ga$, $F$, $\be$, $q$  (see Definition \ref{Def_DG}). There exists a constant  $\si\in (0,1)$   depending only on  $\ga$, $\be$, $q$, such that   for any   $B_{4R}(x_0) \subset \Om $
\begin{equation}\label{Est_osc_1}
\operatorname{osc}\limits_{B_R(x_0) } u \ \le \  \si ~\operatorname{osc}\limits_{B_{4R}(x_0)} u    \, + \,  c_2\, F\, R^{1-\frac nq}
\end{equation}
where $c_2>0$ is a constant from \eqref{maximum_decay_estimate}.
\end{lemma}

\begin{proof}
Define $k_0\in \Bbb R$ by \eqref{Define_k_0} and consider the case:
$$
 \big| \{ \, x\in B_{2R}(x_0): ~u(x) \le k_0  \, \} \big|\ \ge \ \tfrac 12~|B_{2R}|.
$$
Then from \eqref{maximum_decay_estimate} we obtain
$$
M(x_0, R) \ \le \ (1-\theta)\, M(x_0, 4R) + \theta \, k_0 + c_2\, F\, R^{1-\frac nq}.
$$
Subtracting from both sides $m(x_0, 4R)$ we arrive at
\begin{equation}\label{oscillation_decay_estimate}
\om(x_0, R) \, \le \, (1-\tfrac{\theta}2)\, \om(x_0, 4R) \, + \, c_2\, F\, R^{1-\frac nq}.
\end{equation}
In  the second case
$$
 \big| \{ \, x\in B_{2R}(x_0): ~u(x) \ge k_0  \, \} \big|\ \ge \ \tfrac 12~|B_{2R}|
$$
we denote
$
v:= -u$,  $l_0:=-k_0$ and obtain   $v\in DG(\Om, l_0)$ and
 $$
 \big| \{ \, x\in B_{2R}(x_0): ~v(x) \le l_0  \, \} \big|\ \ge \ \tfrac 12~|B_{2R}|.
$$
Then from \eqref{maximum_decay_estimate}
we again arrive at \eqref{oscillation_decay_estimate}.
 \end{proof}

Now we present the  oscillation estimate for    points on the singular curve $\Ga$:

\begin{lemma}\label{osc_lemma_curve}  Assume $\Om\subset \Bbb R^3$ is a bounded domain which contains the origin. Denote by $\dl_1\in (0,1)$ and $\la_1\in (0,1)$ the constants from Proposition \ref{Weak_Harnack}.
  Assume $b\in L_{w}^{2,1}(\Om)$    satisfies $\div b =0$ in $\mathcal D'(\Om)$. Assume         $f\in L_p(\Om)$ with  $p>2$ and  let $u$ be a $p$-weak solution to the problem \eqref{Equation} such that
  $$
  u|_{\Ga } \, = \, 0.
  $$
  Assume  $\al\not =0$, $x_0\in \Ga$ and $B_{4R}(x_0)\subset \Om$.
  There is a constant $\si\in (0,1)$ depending only on $\|b\|_{L_w^{2,1}(\Om)}$  and $\al$  such that   for  any $x_0\in \Ga$ and any   $B_{4R}(x_0) \subset \Om $  we have
\begin{equation}\label{Est_osc_2}
\operatorname{osc}\limits_{B_R(x_0) } u \ \le \  \si ~\operatorname{osc}\limits_{B_{4R}(x_0)} u    \, + \,  c_2\, F\, R^{1-\frac 3q}
\end{equation}
where $c_1>0$ depends only on $n$, $\ga$, $\be$, $q$ and $\al$.
\end{lemma}

\begin{proof}
Assume  \ $u $ is a $p$-weak solution to \eqref{Equation} and $u|_\Ga=0$. In the next section   (see Proposition \ref{Singular_Curve_DG}) we will show that in this case
$ \pm u \in DG(B_{4R}(x_0); 0)$. Define $k_0\in \Bbb R$ and $ k_1\in \Bbb R$ by \eqref{Define_k_0} and \eqref{Define_k_1} respectively   and assume $k_0\ge 0$.
By Proposition \ref{Weak_Harnack_1} either \eqref{What we need} of \eqref{Assumption_f} hold. In the case of \eqref{Assumption_f} we obtain by the H\" older inequality for $q>3$
$$
 \om(x_0, 4R) \, \le \, \frac{2}{c_\star |\al|} \, \| f\|_{L_2(B_{4R}(x_0))} \,  R^{-\frac 12} \, \le \,  \frac{c}{c_\star |\al|} \, \| f\|_{L_q(B_{4R}(x_0))} \, R^{1-\frac 3q}
$$
and hence \eqref{Est_osc_2} follows. Assume now \eqref{What we need} holds.
Note that $k_1\ge k_0\ge 0$ and hence $u\in DG(\Om, k_1)$. Hence from Proposition \ref{Density_lemma} we conclude
 $$
M(x_0, R) \ \le \ (1-\theta)\, M(x_0, 4R) + \theta \, k_1 + c_1\, F\, R^{1-\frac 3q}
$$
and taking into account \eqref{Define_k_1}
we arrive at
\begin{equation}
\om(x_0, R) \, \le \, (1-\tfrac{\la_1\theta } 2 )\, \om(x_0, 4R) \, + \, c_2 \, F\, R^{1-\frac 3q}.
\label{omega_estimate}
\end{equation}
Now consider the case  $k_0   \le  0 $. Denote
$
v:= -u$, $l_0:=-k_0$
\begin{equation}
  l_1 \ := \   -m(x_0, 4R) \, - \,  \frac{\la_0}2  \, \om(x_0, 4R).
  \label{Define_l_1}
\end{equation}
Note that    $v\in DG(\Om; 0)$ and $l_0\ge 0$.  Applying  Proposition \ref{Weak_Harnack_1}  for $v$ we obtain either \eqref{Assumption_f}
or
\begin{equation}\label{What we need_v}
  \big| \{ \, x\in B_{2R}(x_0): ~v(x) \le  l_1  \, \} \big|\ \ge
\ \dl_1 \, |B_{2R}|
  \end{equation}
hold. In the case of \eqref{Assumption_f} we obtain \eqref{Est_osc_2} immediately.  In the case of \eqref{What we need_v} we apply Proposition \ref{Density_lemma} for $v$ and  conclude
 $$
-m(x_0, R) \ \le \ -(1-\theta)\, m(x_0, 4R) + \theta \, l_1 + c_1\, F\, R^{1-\frac 3q}
$$
and taking into account \eqref{Define_l_1}
we arrive at \eqref{omega_estimate} again.
\end{proof}

Finally we present the oscillation estimate near the boundary:

\begin{lemma}\label{osc_lemma_boundary}
Let $\Om\subset \Bbb R^n$ be a bounded Lipschitz domain, $n\ge 2$, and $\pm u\in DG(\Om,  0)$ where $DG(\Om; 0)$ is the De Giorgi class with the parameters  $\ga$, $F$, $\be$, $q$  and the initial level $k_0=0$ (see Definition \ref{Def_DG}). For any $\dl>0$ there exists a constant  $\si\in (0,1)$   depending only on  $\dl$, $\ga$, $\be$, $q$, such that   if for   some  $B_{4R}(x_0) \subset \Om $ the estimate
 \begin{equation}\label{Density_condition_0}
  \big| \{ \, x\in B_{2R}(x_0): ~u(x) =0   \, \} \big|\ \ge
\ \dl \, |B_{2R}|
\end{equation}
 is valid
 then \eqref{Est_osc_1} holds.
\end{lemma}

\begin{proof}
  Define $k_0\in \Bbb R$ by \eqref{Define_k_0} and consider the case $k_0 \ge 0$. Then $$\{ \, x\in B_{2R}(x_0): ~u(x) = 0   \, \} \, \subset \, \{ \, x\in B_{2R}(x_0): ~u(x) \le  k_0   \, \}$$ and we obtain
hence \eqref{Density_condition} holds.
 As $u \in DG(\Om , k_0)$ we obtain
 $$
M(x_0, R) \ \le \ (1-\theta)\, M(x_0, 4R) + \theta \, k_0 + c_1\, F\, R^{1-\frac nq}
$$
and hence
\begin{equation}\label{Leads_to here}
\om(x_0, R) \, \le \, (1-\tfrac\theta2)\, \om(x_0, 4R) \, + \, c_1\, F\, R^{1-\frac nq}.
\end{equation}
In  the case $k_0 \le 0$ we denote
$
v:= -u$, $l_0:=-k_0$.  Then
$l_0\ge 0$ and $v\in DG(\Om, l_0)$. As $$\{ \, x\in B_{2R}(x_0): ~u(x) = 0   \, \} \, \subset \, \{ \, x\in B_{2R}(x_0): ~v(x) \le  l_0   \, \}$$
 we obtain
$$
  \big| \{ \, x\in B_{2R}(x_0): ~v(x) \le  l_0   \, \} \big|\ \ge \ \dl~|B_{2R}|
 $$
 and  hence
 $$
-m(x_0, R) \ \le \ -(1-\theta)\, m(x_0, 4R) + \theta \, l_0 + c_1\, F\, R^{1-\frac nq}
$$
 which again leads to \eqref{Leads_to here}.
\end{proof}

\newpage
\section{H\" older continuity of weak solutions} \label{Proof_T3}
\setcounter{equation}{0}

\bigskip

In this section we show that under assumption on the drift term \eqref{Drift_Specific}, \eqref{Assumption_b}, \eqref{Assumptions_Morrey_space_b}
any $p$-weak solution $u$ to the problem \eqref{Equation} belongs to De Giorgi classes from Section \ref{DG_section}. As a consequence we obtain   the proof of Theorem \ref{Theorem_3}.  First we consider an internal point $x_0$ away from the singular curve $\Ga$.

\begin{proposition}
  \label{Internal_DG}
  Let all assumptions of Theorem \ref{Theorem_3} hold and assume $B_{4R}(x_0)\subset \Om \setminus \Ga$. Then for any $k_0\in \Bbb R$ we have $\pm u\in DG(B_{4R}(x_0);k_0)$  where $DG(\Om; k_0)$ is the De Giorgi class in Definition \ref{Def_DG} with $n=3$,  $F = \| f\|_{L_q(\Om)}$ and some   $\ga>0$, $\be>0$ which depend only on $\| b\|_{L_w^{2,1}(\Om)}$ and $\al$.
\end{proposition}

In the case $B_{4R}(x_0)\subset \Om \setminus \Ga$ the drift $b_0$ is divergence free in $B_{4R}(x_0)$. Hence  Proposition \ref{Internal_DG} follows from  \cite[Section 5]{Chernobai_Shilkin_1}. Here we outline the proof for reader's convenience.

\begin{proof} From Theorem \ref{Theorem_2} we conclude $u\in L_\infty(\Om)$.
Let us fix some $r\in \left(\frac{6}{5},2\right)$ and  $\theta\in \left(\frac 3r -\frac 32,1\right)$. Then from \eqref{Assumptions_Morrey_space_b_0} and Proposition \ref{Holder_inequality} we obtain
$$b_0\in L^{r, 3-r}(\Om), \qquad \| b_0\|_{L^{r, 3-r}(\Om)} \, \le \,  c\, \| b_0\|_{L^{2, 1}_w(\Om)}.$$
  Take  some radius $\rho<R$ and  a cut-off function $\zeta\in C_0^\infty(B_R(x_0))$ such that  \begin{equation}0\le \zeta\le 1 ,  \qquad \zeta\equiv 1 \quad \mbox{on} \quad B_\rho(x_0), \qquad |\nabla \zeta|\, \le \, \frac{c}{R-\rho}. \label{Cut-off} \end{equation}
   Assume $k\in \Bbb R$ is arbitrary and denote
   \begin{equation}\label{u_tilde}
   \tilde u \, := \, (u-k)_+ \, \equiv \, \max\{ u-k, 0\}, \qquad \tilde u \in L_\infty(\Om)\cap W^1_p(\Om).
   \end{equation}
Fix $m:=\frac 1{1-\theta}$ and note that $2m-1=m(1+\theta)$.
   Take $\eta = \zeta^{2m} \tilde u $  in \eqref{Identity}. Taking into account $$\div b_0=0 \quad \mbox{in} \quad \mathcal D'(B_{4R}(x_0))$$  with the help of integration by parts we obtain
  \begin{equation}\label{Arrive_at_this}
  \mathcal B[u, \eta] \ = \   - \   m\, \int\limits_\Om \zeta^{2m-1} \,  b_0\cdot \nabla \zeta  \, |\tilde u |^2\, dx.
    \end{equation}
  As $2m-1=m(1+\theta)$ we obtain
  \begin{equation}\label{From_this}
  \gathered
  \|\zeta^m\nabla\tilde u\|_{L_2(B_R(x_0))}^2 \ \le  \ c\, \| \tilde u \nabla \zeta\|_{L_2(B_R(x_0))}^2  +   m\, \int\limits_\Om \zeta^{m(1+\theta)}\,  b_0\cdot \nabla \zeta  \, |\tilde u|^2\, dx \ +  \\ + \
  \| f\|_{L_q(\Om)}^2 \, |A_k\cap B_R(x_0) |^{1-\frac 2q}
  \endgathered
  \end{equation}
  where  $A_k:= \{\, x\in \Om: \, u(x)>k\, \}$.   Taking into account  \eqref{Cut-off} and applying the estimate \eqref{Morrey_Compactness_estimate}   we obtain
  $$
  \int\limits_\Om \zeta^{m(1+\theta)}\,  b\cdot \nabla \zeta  \, |\tilde u|^2\, dx  \ \le \  \frac{c\, R^\theta}{ R-\rho  }\, \|b_0\|_{L^{r, 3-r}(\Om)} \, \left\| \nabla (\zeta^m \tilde u )\right\|_{L_2(B_R(x_0))}^{1+\theta }
  \| \tilde u\|_{L_2(B_R(x_0))}^{1-\theta }.
  $$
  Taking arbitrary $\ep>0$ and applying the Young inequality we obtain
  $$
  \gathered
  \int\limits_\Om \zeta^{m(1+\theta)} b\cdot \nabla \zeta  \, |\tilde u |^2\, dx \ \le \ \ep \, \left\| \nabla \left(\zeta^m \tilde u\right)\right\|_{L_2(B_R(x_0))}^2 \ + \\ +  \ \frac{c_\ep}{(R-\rho)^2} \Big( \frac R{R-\rho}\Big)^{\frac {2\theta} {1-\theta}}\,   \|b_0\|_{L^{r, 3-r}(B_R(x_0))}^{\frac { 2} {1-\theta}}     \|   \tilde u \|_{L_2(B_R(x_0))}^2.
  \endgathered
  $$
  So, if we fix sufficiently small $\ep>0$ from \eqref{From_this} for any $k\in \Bbb R$ and $0<\rho<R$ we obtain
  \begin{equation}\label{Gives_this}
  \gathered
 \tfrac 12\,  \| \nabla (u-k)_+\|_{L_2(B_\rho(x_0))}^2 \ \le \\ \le  \ \frac{c}{(R-\rho)^2}\, \left(1+ \Big( \frac R{R-\rho}\Big)^{\frac {2\theta} {1-\theta}}\,   \|b_0\|_{L^{r, 3-r}(\Om)}^{\frac { 2} {1-\theta}} \right)\,  \| (u-k)_+  \|_{L_2(B_R(x_0))}^2    + \\ + \
  \| f\|_{L_q(\Om)}^2 \, |A_k\cap B_R(x_0) |^{1-\frac 2q}.
  \endgathered
  \end{equation}
  Hence we obtain that $u\in DG(\Om)$. Applying the same arguments to $-u$ instead of $u$ we also obtain $-u\in DG(\Om)$.
\end{proof}

 Now  we consider an internal  point $x_0$ laying on  the singular curve $\Ga$.

\begin{proposition}
  \label{Singular_Curve_DG}
  Let all assumptions of Theorem \ref{Theorem_3} hold and assume $x_0\in \Ga$ and $B_{4R}(x_0)\subset \Om $. Then for any $k_0\ge 0 $ we have $\pm u\in DG(B_{4R}(x_0); 0)$  where $DG(\Om;  0)$ is the De Giorgi class in Definition \ref{Def_DG} with $n=3$, $k_0=0$, $F = \| f\|_{L_q(\Om)}$ and some   $\ga>0$, $\be>0$ which depend only on $\| b\|_{L_w^{2,1}(\Om)}$ and  $\al$.
\end{proposition}

\begin{proof} We take arbitrary $k\ge 0$ and   proceed as in the proof of Proposition \ref{Internal_DG}. Define $\tilde u$ by \eqref{u_tilde}. As $u|_{\Ga} =0$ and $k\ge 0$ we conclude
$  \tilde u|_{\Ga} = 0$ in the sense of traces and from Proposition \ref{Q_Form} we conclude
\begin{equation}\label{Q_form_vanishes}
 \mathcal B[\zeta^m\tilde u, \zeta^m \tilde u ] \ = \ 0.
\end{equation}
Hence we again arrive at \eqref{Arrive_at_this} and proceed in the same way as in Proposition \ref{Internal_DG}.
\end{proof}

 Finally   we consider a   point $x_0$ laying on  the boundary $\cd \Om$. Note that as $\Om$ is a bounded Lipschitz domain there exist $R_*>0$ and $\dl_*$  such that for any $x_0\in \cd \Om$ and any $R< R_*$
 \begin{equation}\label{Boundary_density}
 |B_R(x_0)\setminus \Om| \, \ge \, \dl_* \, |B_R|.
 \end{equation}

\begin{proposition}
  \label{Boundary_DG}
  Let all assumptions of Theorem \ref{Theorem_3} hold and denote by $\bar u$ the zero extension of $u$ outside $\Om$. Assume $x_0\in \cd \Om$ and $4R\le R_*$. Then for any $k_0\ge 0 $ we have $\pm \bar u\in DG(B_{4R}(x_0); 0)$  where $DG(\Om;  0)$ is the De Giorgi class in Definition \ref{Def_DG} with $n=3$, $k_0=0$, $F = \| f\|_{L_q(\Om)}$ and some   $\ga>0$, $\be>0$ which depend only on $\| b\|_{L_w^{2,1}(\Om)}$ and  $\al$.
\end{proposition}

\begin{proof}

Denote $  \Om_R(x_0):=\Om \cap B_R(x_0)$.  Take a cut-off function $\zeta\in C_0^\infty(B_R(x_0))$ satisfying \eqref{Cut-off}. Then for any $k\ge 0$   the function  $\eta = \zeta^{2m} (  u-k)_+$  vanishes on $\cd \Om$ hence it is admissible for the identify \eqref{Identity}. From $ (  u-k)_+|_{\Ga}=0$ we conclude \eqref{Q_form_vanishes} holds. Proceeding as in the proofs of Propositions \ref{Internal_DG}, \ref{Singular_Curve_DG} we arrive at
  $$
  \gathered
 \tfrac 12\,  \| \nabla (u-k)_+\|_{L_2(\Om_\rho(x_0))}^2 \ \le \\ \le   \ \frac{c}{(R-\rho)^2}\, \left(1+ \Big( \frac R{R-\rho}\Big)^{\frac {2\theta} {1-\theta}}\,   \|b\|_{L^{r, 3-r}(\Om)}^{\frac { 2} {1-\theta}}\right)\,  \| (u-k)_+  \|_{L_2(\Om_R(x_0))}^2    + \\ + \
  \| f\|_{L_q(\Om)}^2 \, |A_k\cap \Om_R (x_0)|^{1-\frac 2q},
  \endgathered
  $$
  which  gives \eqref{Gives_this} with the function $\bar u$ instead of $u$. Hence we obtain $\bar u \in DG(B_{4R}(x_0))$. Similarly we obtain $-\bar u \in DG(B_{4R}(x_0))$.
\end{proof}

Now we can prove Theorem \ref{Theorem_3}. Taking into account \eqref{Boundary_density} we can iterate estimates in Proportions  \ref{osc_lemma}, \ref{osc_lemma_curve}, \ref{osc_lemma_boundary} and obtain the oscillation estimate
\begin{equation}\label{Decay_estimate}
 \forall\, \rho< R \qquad
 \operatorname{osc}\limits_{B_\rho(x_0)\cap \Om} u
\, \le \, c_2\, \left( \Big(\frac \rho {R}\Big)^\mu  \, \| u\|_{L_\infty(\Om)}  + F\, \rho^\mu\right)
\end{equation}
 with some $\mu\in (0,1)$ depending only on $\si\in (0,1)$ and $q>3$ in one of the following three cases:
 \begin{itemize}
   \item[(a)] $B_{4R}(x_0)\subset \Om\setminus \Ga$,
   \item[(b)] $x_0\in\Ga$, $B_{4R}(x_0)\subset \Om$,
   \item[(c)] $x_0\in \cd \Om$, $R<\frac 14 R_*$.
 \end{itemize}
 Then the inequality \eqref{Decay_estimate} for an arbitrary $x_0\in \bar \Om$ and $R<\frac 14 R_*$ can be obtained by a standard combination of inequalities (a), (b), (c). From
 this inequality and  \eqref{L_infty_Estimate} the estimate \eqref{Holder_Estimate} follows immediately. Theorem \ref{Theorem_3} is proved.

\newpage
\section{Existence and uniqueness of $p$-weak solutions} \label{Proof_T1}
\setcounter{equation}{0}

\bigskip In this section we prove Theorem \ref{Theorem_1}.
First   we establish the higher integrability of weak solutions to the problem \eqref{Equation}.

\medskip
\begin{proposition}\label{L_p_estimate}   Assume   $b\in L_{w}^{2,1}(\Omega)$  satisfies \eqref{Drift_Specific}, \eqref{Assumption_b} and assume $f\in L_q(\Om)$ with $q>3$. Then there exists $p>2$ depending only on $q$,  $\al$, $\|b\|_{L_{w}^{2,1}(\Omega)}$ and the Lipschitz constant of $\cd \Om$ such that for any   $p$-weak solution $u$  of the problem \eqref{NSE} satisfying additional assumption  \eqref{Zero_on_Gamma} the   estimate \eqref{Main_Estimate}
 holds with some constant $c>0$ depending only on $q$, $\al$ and $\|b\|_{L_{w}^{2,1}(\Omega)}$ and the Lipschitz constant of $\cd \Om$.
\end{proposition}

\begin{proof}
  Assume $p>2$ and let $u$ be a $p$-weak solution to \eqref{NSE}. Then we can interpret $u$ as a $p$-weak solution to the problem
  \begin{equation}\label{Equation_div-free}
    \left\{ \quad \gathered -\Delta u+b\cdot\nabla u \, = \, \div g \quad \mbox{in} \quad \Om, \\ u|_{\cd \Om} \, = \, 0 \qquad
    \endgathered \right.
  \end{equation}
  with the right hand side
  $$
  g \,  = \, f \, + \, \al \,  \frac{x'}{|x'|^2} \, u.
  $$
  and hence for any   $p >2$ we have
  $$
  \| g\|_{L_p(\Om)} \, \le \, \|f\|_{L_p(\Om)} \, + \, c\, \left\| \tfrac{u}{|x'|}\right\|_{L_p(\Om)}
  $$
  From Theorem \ref{Theorem_3} we obtain $u\in C^\mu(\bar \Om)$ with some $\mu \in (0,1)$ depending only on
   $\Om$, $\al$,  $q$   and $\|b\|_{L^{2, 1}_{w}(\Om)} $. Fix some $p_0\in \left( 2, \frac{2}{1-\mu}\right)$ and denote $q_0= \min\{ q, p_0\}$.
  Taking into account \eqref{Zero_on_Gamma}  we obtain
  $$
  \left\| \tfrac{u}{|x'|}\right\|_{L_{q_0}(\Om)} \, \le \, c\, \| u\|_{C^\mu (\bar \Om)}
  $$
  and hence from Theorem \ref{Theorem_3}  we arrive at
  $$
  \| g\|_{L_{q_0}(\Om)} \, \le \, \|f\|_{L_q(\Om)} \, + \, c\, \| u\|_{C^\mu(\bar \Om)} \, \le \, c\, \|f\|_{L_q(\Om)}
  $$
   Since $\div b=0$,  we conclude $u$ is a $p$-weak solution of  the problem \eqref{Equation_div-free} with a divergence-free drift $b$ and the right-hand side $q\in L_{q_0}(\Om)$ with some $q_0>2$. Hence from   \cite{Chernobai_Shilkin_1}  we obtain there exists $p\in (2, q_0)$ such that
  $$
  \|u\|_{W^1_p(\Omega)} \, \le \,  c\,  \|g\|_{L_{p}(\Omega)}
  $$
  from which we obtain \eqref{Main_Estimate}.
\end{proof}

\medskip
Now we turn to the proof of Theorem \ref{Theorem_1}.
We follow the method   we used in  \cite{Chernobai_Shilkin} in the 2D case. As a first step  we construct a solution of an auxiliary problem satisfying the ``non-spectral'' condition \eqref{Assumptions_b_good_sign}.

\begin{proposition}\label{prop7.1}
Assume $\al<0$. Then  there exists  $p_1>2$ depending only on $\Om$ and $\al$ such that for any  $f\in C_0^\infty(\Om\setminus \Ga)$  there exists  a unique $p_1$--weak solution  $v$ to the problem
\begin{equation}\label{support_system}
\left\{ \quad \gathered -\Delta v     \, - \,  |\al|\,   \frac{x'}{|x'|^2} \cdot \nabla v  \ = \  -|x'|^{\al}\div f  \quad \mbox{in} \quad \Om, \\ v|_{\cd \Om} \ = \ 0. \qquad \qquad \endgathered \right.
\end{equation}
Moreover, $v$ is H\" older continuous in $\bar \Om$.
\end{proposition}

Proposition \ref{prop7.1} is proved in  \cite{Chernobai_Shilkin_1}.

\medskip
Now we apply the so-called ``Darboux transform'' to the function $v$ to construct a solution of an auxiliary  problem with $\al<0$ (which corresponds to the ``spectral'' case  \eqref{Assumptions_b_bad_sign}) and vanishing on $\Ga$.

\begin{proposition}\label{prop7.2}
Assume $\al<0$ and $q>3$. Then  there exists  $p>2$ depending only on  $q$, $\al$ and the Lipschitz constant of $\cd \Om$  such that for any  $f\in C_0^\infty(\Om\setminus \Ga)$  there exists  a unique $p$--weak solution  $u$ to the problem
\begin{equation}
\left\{ \quad \gathered     -\Delta u -  \al\, \frac{x'}{|x'|^2} \cdot \nabla u \ = \ -\div f \qquad\mbox{in}\quad \Om, \\
u|_{\cd \Om} \ = \ 0, \qquad \qquad
\endgathered\right.
\label{Equation_h}
\end{equation}
which satisfies the condition  \eqref{Zero_on_Gamma}.
Moreover, $u$ is H\" older continuous in $\bar \Om$ and satisfies   estimates  \eqref{Holder_Estimate} and \eqref{Main_Estimate}.

\end{proposition}

\begin{proof}
Let     $v$ be a $p_1$-weak solution of \eqref{support_system}.
Denote
$$
u(x) \, = \, |x'|^{|\al|} \,  v(x)
$$
If $|\al|\ge 1$ then the function $|x'|^{|\al|}$ is Lipschitz continuous. Hence we obtain $u\in W^1_{p_1}(\Om)$. In the case $0<|\al|<1$ we have $u\in W^1_{p }(\Om)$ for any $p $ satisfying
\begin{equation}\label{p_depends_on_alpha}
2< p  \, < \, \min\left\{\, p_1, \, \frac{2}{1-|\al|} \, \right\}.
\end{equation}
In any case we obtain   $u\in \overset{\circ}{W}{^1_{p }}(\Om)$ for some $p >2$, $u$ is H\" older continuous on $\bar \Om$ and satisfies the condition \eqref{Zero_on_Gamma}. Let us verify $u$ is a $p $-weak solution to \eqref{Equation_h}. Indeed,  taking arbitrary $\eta \in C_0^\infty(\Om)$ and testing \eqref{support_system} by $|x'|^{|\al|}\eta $ with the help of  identities
$$
\nabla \big(|x'|^{|\al|} \eta\big) \, = \,  |x'|^{|\al|} \Big(\nabla \eta + |\al|\, \frac {x'}{|x'|^2}\, \eta\Big), \qquad  |x'|^{|\al|} \nabla v  \ = \ \nabla u  + \al  u \frac{x'}{|x'|^2}
$$
we arrive at the identity
$$
\int\limits_\Om \Big( \nabla u + \al\,  u \frac{x'}{|x'|^2}\Big) \cdot \nabla \eta\, dx \ = \ \int\limits_\Om f\cdot \nabla \eta\, dx , \qquad \forall \, \eta\in C_0^\infty(\Om).
$$
Taking into account   $u|_\Ga =0$ and using integration by parts we obtain
$$
\al\, \int\limits_\Om \frac{x'}{|x|^2}\cdot \nabla u \, \eta \, dx  \ = \ -\al \,    \int\limits_\Om \frac{x'}{|x|^2}\cdot \nabla \eta  \, u \, dx
$$
which gives
$$
- \Delta u -\al \frac{x'}{|x'|^2}\cdot \nabla u   \ = \ - \div f \qquad \mbox{in} \quad \mathcal D'(\Om).
$$
Now we can fix $\mu\in (0,1)$ and $p>2$ as in Theorem \ref{Theorem_3} and    Proposition \ref{L_p_estimate}. Without loss of generality we can assume \eqref{p_depends_on_alpha} is satisfied. Then from Theorem \ref{Theorem_3} and  in Proposition \ref{L_p_estimate} we obtain inequalities \eqref{Holder_Estimate} and \eqref{Main_Estimate}. Note that  $\mu\in (0,1)$ and $p>2$   depend only on $\Om$, $q$ and $\al$.
\end{proof}

Now  we can relax the assumption on the smoothness of the right hand side $f$.

\begin{proposition}\label{prop7.3}
 Assume $\al<0$, $q>3$ and assume  $\mu\in (0,1)$ is defined in Theorem \ref{Theorem_3}  and $p>2$ is defined in Proposition \ref{prop7.2}. Then for any  $f\in L_q(\Om)$  there exists a unique $p$-weak solution $u$ of the system \eqref{Equation_h} which satisfies  \eqref{Zero_on_Gamma}. Moreover, $u\in C^\mu(\bar \Om)$    and the estimates \eqref{Holder_Estimate}, \eqref{Main_Estimate} hold.
\end{proposition}

\begin{proof}
  Assume $f\in L_q(\Om)$ and take  $f_\ep\in C^{\infty}_0(\Omega\setminus\Ga)$  so that $\|f_\ep-f\|_{L_q(\Omega)} \to 0$ and $\ep\to 0$.
   From Proposition \ref{prop7.2} we obtain the existence of $p$-weak solutions  $\{u_\ep\}$  to the problem \eqref{Equation_h} with right hand side $f_\ep$ Moreover, using  estimates \eqref{Holder_Estimate} and \eqref{Main_Estimate} we obtain inequality
  $$
  \|u_\ep\|_{W^1_p(\Om)}  \, +\, \| u_\ep \|_{C^\mu(\bar \Om)} \, \le \, c\, \|f_\ep\|_{L_q(\Om)}.
  $$
  Hence we can extract a subsequence such that $u_\ep\rightharpoonup u$ in $W^1_q(\Om)$ and $u_\ep\to u$ uniformly in $\bar \Om$. It is easy to check that $u$ will satisfy \eqref{Equation_h} with right hand side $f$ and   \eqref{Zero_on_Gamma} holds.
\end{proof}

Now we can prove the existence of $p$-weak solutions to the problem \eqref{Equation} in the case of  a smooth divergence free part of the drift.
\begin{proposition}\label{prop7.4}
 Assume $\al<0$, $q>3$ and  $b\in C^{\infty}(\Om)$ satisfies $\div b=0$. Then there exist $\mu\in (0,1)$ and $p>2$ depending only on  $q$, $\al$, $\|b\|_{L_w^{2,1}(\Om)}$ and the Lipschitz constant of $\cd \Om$  such that
 for any $f\in L_q(\Om)$ there exists a unique $p$-weak solution $u$ to the problem \eqref{Equation} which  satisfies the condition \eqref{Zero_on_Gamma}. Moreover,
 $u\in C^\mu(\Om)$ and the estimates \eqref{Holder_Estimate} and \eqref{Main_Estimate} hold.
\end{proposition}

\begin{proof}
Take  any $v\in L_q(\Om)$. From Proposition \ref{prop7.3} we obtain the existence the unique $p$-weak solution  $u^v$  to the problem
   \begin{equation}
\left\{ \quad \gathered     -\Delta u^v -  \al\, \frac{x'}{|x'|^2} \cdot \nabla u^v \ = \ \div (f-bv) \qquad\mbox{in}\quad \Om, \\
u^v|_{\cd \Om} \ = \ 0.
\endgathered\right.
\end{equation}
such that $$u^v|_\Ga =0.$$ Moreover, $u^v\in C^\mu (\bar \Om)$ and the estimate
\begin{equation}\label{A_priori_Schauder}
  \|u^v \|_{W^1_p(\Om)} \, +\, \| u^v \|_{C^\mu(\bar \Om)} \, \le \, c\, \Big(\|f\|_{L_q(\Om)} \, + \, \|b\|_{L_\infty(\Om)}\| v\|_{L_q(\Om)}\Big)
\end{equation}
holds.  Define the operator $A: L_q(\Om)\to L_q(\Om)$, $A(v):= u^v$. Applying Theorem \ref{Theorem_3} and Proposition \ref{L_p_estimate} for any $v_1$, $v_2\in L_q(\Om)$ we obtain the inequality
 $$
  \|A(v_1)-A(v_2)\|_{W^1_p(\Omega)} \, + \,  \|A(v_1)-A(v_2)\|_{C^\mu(\bar \Omega)}    \, \le \, c\, \|b\|_{L_\infty(\Omega)} \|v_1-v_2\|_{L_q(\Omega)}
  $$
  which implies the operator $A: L_q(\Om)\to L_q(\Om)$ is continuous. Moreover, from  \eqref{A_priori_Schauder}
  taking into account
  compactness of the imbedding $W^1_p(\Om)\hookrightarrow L_p(\Om)$  it is easy to see  that the operator $A: L_q(\Om)\to L_q(\Om)$ is compact. Hence we can apply the Leray-Shauder fixed point theorem.
Assume $\la\in [0,1]$ and $v\in L_q(\Om)$ satisfies $v=\la A(v)$.
Denote $u:=A(v)$. Then $u$ is  a unique  $p$-weak solution to the
problem
$$
\left\{ \quad \gathered  -\Delta u  +\Big( \la b-\al \, \frac{x}{|x|^2}\Big)\cdot \nabla u \ = \ -\div f \qquad\mbox{in}\quad \Om, \\
u|_{\cd \Om} \ = \ 0,
\endgathered\right.
 $$
 which satisfies the condition \eqref{Zero_on_Gamma}. Hence from   Theorem \ref{Theorem_3} and Proposition \ref{L_p_estimate}  we obtain the estimate
 $$
 \| u\|_{W^1_p(\Om)} \, +\, \| u  \|_{C^\mu(\bar \Om)} \, \le \, c\, \| f\|_{L_q(\Om)}
 $$
  with some constant $c$ depending only on  $\Om$, $q$, $\al$ and $\|  b\|_{L_{2,w}(\Om)}$ and independent of $\la\in [0,1]$. Hence there
 exists   $u\in L_q(\Om)$  satisfying $u=A(u)$.
\end{proof}

Now we can prove Theorem \ref{Theorem_1}.

\begin{proof}
We follow the method similar to \cite[Theorem 2.1]{Kwon_1},  and \cite{Verchota_1},\cite{Verchota_2}. Let us fix some $q>3$ and let
 $\mu\in (0,1)$ and $p>2$ be the constants determined in  Theorem \ref{Theorem_3}   and  Proposition \ref{Theorem_1} respectively.
Denote $p'=\frac p{p-1}$.
As $\Om$ is Lipschitz we can find  the sequence of $C^2$-smooth domains
$\{ \Om_k\}_{k=1}^\infty$   such that $$\Om_{k+1}\Subset \Om_k, \qquad \bigcup\limits_k \Om_k = \Om.$$
Moreover,  it is possible to  construct   domains $\Om_k$ so that  the Lipschitz constants of $\cd \Om_k$ are controlled uniformly by the Lipschitz constant of $\cd \Om$. In particular, we can assume there are exist positive constants $\hat \dl_0$ and $\hat R_0$ independent on $k\in \Bbb N$ such that
\begin{equation}\label{Oblast'_k}
\forall\, k\in \Bbb N, \qquad  \forall\, R<\hat R_0, \qquad \forall\, x_0\in \cd \Om_k \qquad |B_R(x_0) \setminus \Om_k | \, \ge \,   \hat  \dl_0 \, |B_R|. 
\end{equation}
For a canonical domain given as a subgraph of a Lipschitz function existence of such approximation can be obtained by mollification and shift of the graph, and for general bounded Lipschitz domain  the standard localization works.

Now we take a sequence of positive numbers $\ep_k\to 0$ such that   $\ep_k < \operatorname{dist}\{ \bar \Om_k, \cd \Om\}$, and define the mollification of the drift $b$:
  $$
  b_k (x )  \, := \, \int\limits_{\Om} \om_{\ep_k}(x-y)b(y)\, dy, \qquad x\in \Om,
$$
where $\om_\ep(x):=\ep^{-n}\om(x/\ep)$ and $\om\in C_0^\infty(\Bbb R^n)$ is the standard Sobolev kernel, i.e.
\begin{equation}
\om \, \ge \, 0, \qquad \supp \om \in \bar B, \qquad \int\limits_{\Bbb R^n}\om(x)\, dx \, = \, 1, \qquad \om(x)=\om_0(|x|).
\label{Sobolev_kernel}
\end{equation}
Then $b_k \in C^\infty(\bar \Om)$ and as $\ep_k<\operatorname{dist}\{ \bar \Om_k, \cd \Om\}$ from $\div b=0$ in $\mathcal D'(\Om)$   for any $k $ we obtain
\begin{equation}\label{Assumptions_b_k}
\div  b_k \, = \, 0 \qquad\mbox{in} \quad \Om_k.
\end{equation}
Moreover, from \cite[Proposition 6.1]{Chernobai_Shilkin_1} we obtain there is  a constant $c>0$ independent on $k$ such that
\begin{equation}\label{Drift_approximation1}
\| b_k\|_{L^{2,1}_w(\Om_k)} \, \le \, c\, \|b\|_{L^{2,1}_w(\Om)}, \qquad \| b_k -b\|_{L_{p'}(\Om)} \to 0 \quad \mbox{as} \quad k \to 0.
\end{equation}
For  $f\in L_q(\Om)$  we can find    $f_k\in C^\infty_0(  \Om_k )$ such that $\|f_k-f\|_{L_q(\Om)}\to 0$. From Proposition \ref{prop7.4} we conclude that for any $k\in \Bbb N$ there exists  a unique $p$-weak solution $u_k\in W^1_p(\Om)\cap C^\mu(\bar \Om)$  to the  problem
  \begin{equation}\label{Equation_Approx}
\left\{ \quad \gathered     -\Delta u_k +  \Big( b_k -\al \, \frac{x'}{|x'|^2}\Big) \cdot \nabla u_k \ = \ -\div f_k \qquad\mbox{in}\quad  \Om_k , \\
u_k|_{\cd \Om_k}  \ = \ 0,
\endgathered\right.
\end{equation}
which satisfies  the condition
$$
u_k|_{\Ga} = 0
$$
in the sense of traces.
Extend functions $u_k$ by zero from $\Om_k$ onto $\Om$. From Proposition \ref{L_p_estimate} we obtain the estimate
$$
\| u_k \|_{W^1_p(\Om)} \ \le \ c\, \| f_k\|_{L_p(\Om)}
$$
with a constant  $c>0$ depending only on  $q$,  $\| b\|_{L^{2, 1}_w(\Om)}$  and the constant $\hat \dl_0$ in \eqref{Oblast'_k} which is independent on $k$.  Hence we can take a subsequence $u_k$ such that
$$
u_k \rightharpoonup u \quad \mbox{in}\quad W^1_p(\Om).
$$
As for $p>2$  the trace operator is compact from $W^1_p(\Om)$ into $L_p(\Ga)$ we obtain
$$
u|_{\Ga} = 0
$$
in the sense of traces.
 Take any $\eta\in C_0^\infty(\Om)$, due to our construction of   $\Omega_k$   for sufficiently large $k$ we have $\supp \eta\subset\Om_k$ and hence   $\eta$ is a suitable test function in \eqref{Equation_Approx}.  As $b_k \to b$ in $L_{p'}(\Om)$ we can pass to the limit in the  identity
$$
\gathered  \int\limits_{\Om}   \nabla u_k \cdot\left(
\nabla\eta + \Big( b_k -\al\, \frac{x'}{|x'|^2}\Big)\eta \right) \, dx   \ = \
\int\limits_{\Om} f_k \cdot \nabla \eta~dx   ,
\endgathered
\label{Identity_eps}
$$
and obtain \eqref{Identity}. Hence $u \in \overset{\circ}{W}{^1_p}(\Om)$ is a $p$-weak solution to the problem \eqref{Equation} satisfying \eqref{Zero_on_Gamma} and \eqref{Main_Estimate}.
From Theorem \ref{Theorem_3} we obtain $u\in C^\mu(\bar \Om)$ and the estimate \eqref{Holder_Estimate}.
The  uniqueness of $u$ follows from the estimate \eqref{Main_Estimate}.
\end{proof}

\end{document}